\documentclass{amsart}
\usepackage[dvipsnames]{xcolor}
\usepackage{amssymb}
\usepackage[utf8]{inputenc}
\usepackage{tikz}
\usepackage{tikz-cd}
\usepackage{amsmath, amsthm, amssymb, amsfonts}
\usepackage{hyperref}
\usepackage{mathtools}
\usepackage[retainorgcmds]{IEEEtrantools}
\usepackage{latexsym}
\usepackage{enumerate}
\usepackage{nomencl}
\usepackage{nameref}
\usepackage{chngcntr}
\usepackage{bbm}
\usepackage[margin=1.6in]{geometry}

\colorlet{mylinkcolor}{Blue}
\colorlet{mycitecolor}{purple}
\colorlet{myurlcolor}{Emerald}

\hypersetup{
  linkcolor  = mylinkcolor!,
  citecolor  = mycitecolor!,
  urlcolor   = myurlcolor!,
  colorlinks = true,
}

\usetikzlibrary{fit, patterns}

\usetikzlibrary{decorations.pathmorphing}

\tikzset{%
    symbol/.style={%
        draw=none,
        every to/.append style={%
            edge node={node [sloped, allow upside down, auto=false]{$#1$}}}
    }
}

\numberwithin{equation}{section}

\newtheorem*{thma}{Theorem}

\newtheorem*{cora}{Corollary}

\newtheorem{thm}{Theorem}[section]
\newtheorem{lem}[thm]{Lemma}

\newtheorem{cor}[thm]{Corollary}

\theoremstyle{definition}
\newtheorem{defi}[thm]{Definition}

\theoremstyle{remark}
\newtheorem{rem}[thm]{Remark}

\newcommand{\Rest}{\operatorname{Res}}
\newcommand{\op}{\normalfont{^{op}}}

\newcommand{\Mod}[1]{#1\operatorname{-mod}}

\newcommand{\MOD}[1]{#1\operatorname{-Mod}}

\newcommand{\Ab}{\operatorname{Ab}}
\newcommand{\Hom}[3]{\operatorname{Hom}_{#1}\left( #2,#3\right)  }

\newcommand{\End}[2]{\operatorname{End}_{#1}\left(#2\right) }

\newcommand{\Ker}[1]{\operatorname{Ker}#1}

\newcommand{\Ext}[4]{\operatorname{Ext}_{#1}^{#2}\left(#3,#4 \right)}

\newcommand{\Z}{\mathbb{Z}}

\newcommand{\Znn}{\Z_{\geq 0}}
\newcommand{\Zp}{\Z_{> 0}}

\tikzset{%
    symbol/.style={%
        draw=none,
        every to/.append style={%
            edge node={node [sloped, allow upside down, auto=false]{$#1$}}}
    }
}

\makeatletter
\@namedef{subjclassname@2020}{%
  \textup{2020} Mathematics Subject Classification}
\makeatother

\begin{document}

\title[All exact Borel subalgebras and all directed bocses are normal]{All exact Borel subalgebras and all directed bocses are normal}
\author{Teresa Conde}
\address{Institute of Algebra and Number Theory, University of Stuttgart\\ Pfaffenwaldring 57, 70569 Stuttgart, Germany}
\email{\href{mailto:tconde@mathematik.uni-stuttgart.de}{\nolinkurl{tconde@mathematik.uni-stuttgart.de}}}
\date{\today}
\thanks{This work was supported by the Deutsche Forschungsgemeinschaft (DFG) through the grant KO 1281/18.}

\begin{abstract}
	Recently, Brzeziński, Koenig and K\"ulshammer have introduced the notion of normal exact Borel subalgebra of a quasihereditary algebra. They have shown that there exists a one-to-one correspondence between normal directed bocses and quasihereditary algebras with a normal and homological exact Borel subalgebra. In this short note, we prove that every exact Borel subalgebra is automatically normal. As a corollary, we conclude that every directed bocs has a group-like element. These results simplify Brzeziński, Koenig and K\"ulshammer's bijection.
\end{abstract}

\maketitle

\section{Introduction}
The concept of exact Borel subalgebra of a quasihereditary algebras was introduced by Koenig in \cite{MR1362252}. Exact Borel subalgebras emulate the properties of Lie-theoretical Borel subalgebras. Not every quasihereditary algebra has an exact Borel subalgebra (\cite[Example 2.3]{MR1362252}, \cite[Appendix A.3]{MR3228437}). However, Koenig, K\"ulshammer and Ovsienko proved in \cite{MR3228437} that every quasihereditary algebra is Morita equivalent to some quasihereditary algebra admitting an exact Borel subalgebra.

Recently, Brzeziński, Koenig and K\"ulshammer introduced in \cite{doi:10.1112/blms.12331} the notions of normal, homological and regular exact Borel subalgebra to describe desirable features of exact Borel subalgebras. In fact, these three definitions are modelled after key aspects of the exact Borel subalgebras studied by Koenig, K\"ulshammer and Ovsienko in \cite{MR3228437}.

 An exact Borel subalgebra $B$ of a quasihereditary algebra $(A,\Phi,\unlhd)$ is \emph{normal} if the inclusion of $B$ into $A$ splits as a morphism of right $B$-modules and the left inverse of the inclusion can be chosen so that its kernel is a right ideal of $A$. Our main goal is to prove the following result.
\begin{thma}[Theorem \ref{thm:main}]
	Every exact Borel subalgebra $B$ of a quasihereditary algebra $(A,\Phi,\unlhd)$ is normal.
\end{thma}
Normality is therefore not a special feature of certain well-behaved exact Borel subalgebras, but rather an attribute shared by all exact Borel subalgebras.

Exact Borel subalgebras are closely related to another class of algebraic objects, known as directed bocses (or directed corings). The right algebra $R_{\mathfrak{B}}$ of a directed bocs $\mathfrak{B}=(B,W,\mu,\varepsilon)$ is quasihereditary and the underlying algebra $B$ embeds into $R_{\mathfrak{B}}$ as an exact Borel subalgebra. This connection between directed bocses and exact Borel subalgebras was first uncovered in \cite{MR3228437} and it was thoroughly studied in \cite{doi:10.1112/blms.12331}. Using techniques from \cite{doi:10.1112/blms.12331}, we deduce the following corollary of our main result.

\begin{cora}[Corollary \ref{cor:main}]
	Every directed bocs $\mathfrak{B}=(B,W,\mu,\varepsilon)$ is normal (i.e.~there exists $w\in W$ satisfying $\varepsilon (w)=1_B$ and $\mu(w)=w\otimes_B w$).
\end{cora}

As a highlight in \cite{doi:10.1112/blms.12331}, Brzeziński, Koenig and K\"ulshammer provided a bijection between (normal) directed bocses and quasihereditary algebras with a (normal) homological exact Borel subalgebra. They also proved that the bijection restricts to a one-to-one correspondence between regular normal directed bocses and quasihereditary algebras with a regular normal exact Borel subalgebra. The results proved in this note lead to a simplification and to a better understanding of bijection in \cite{doi:10.1112/blms.12331} (see Theorem \ref{thm:mainlast}).

\subsection{Notation and conventions}
Throughout this note, $K$ will denote a field. By default, the word ‘algebra’ will mean a finite-dimensional $K$-algebra. All modules are assumed to be finite-dimensional left modules unless stated otherwise. 

Given an algebra $A$, we shall denote the category of finite-dimensional left $A$-modules by $\Mod{A}$; this is a full subcategory of the category $\MOD{A}$ of all (possibly infinite-dimensional) left $A$-modules. The isomorphism classes of the simple $A$-modules may be labelled by the elements of a finite set $\Phi$. Denote the simple $A$-modules by $L_i$ or $L_i^A$ for $i\in \Phi$ and use the notation $P_i$ or $P_i^A$ (resp.~$Q_i$ or $Q_i^A$) for the projective cover (resp.~injective hull) of $L_i$. 

Finally, the category of abelian groups will be denoted by $\Ab$ and the letter $D$ will be used for the standard duality functor $\Hom{K}{-}{K}$.
\section{Background}

\label{sec:background}

We start by providing some background on quasihereditary algebras, exact Borel sub\-algebras and bocses.

\subsection{Quasihereditary algebras}

Assume that $(\Phi,\unlhd)$ is an indexing poset for the simple modules over an algebra $A$. Denote by $\Delta_i$ or $\Delta_i^A$ the largest quotient of the projective indecomposable $P_i$ whose composition factors are all of the form $L_j$ with $j\unlhd i$. Call $\Delta_i$ the \emph{standard module} with label $i\in\Phi$. Dually, denote the \emph{costandard module} with label $i$ by $\nabla_{i}$ or $\nabla_i^A$, i.e.~let $\nabla_{i}$ be the largest submodule of $Q_i$ with all composition factors of the form $L_j$, with $j\unlhd i$. 

\begin{defi}
\label{defi:qh}
An algebra $A$ is \emph{quasihereditary} with respect to a poset $(\Phi, \unlhd)$ indexing all pairwise nonisomorphic simple $A$-modules if the following conditions hold for every $i \in \Phi$:
\begin{enumerate}
\item $L_i$ has multiplicity one as a composition factor of $\Delta_i$;
\item $P_i$ is filtered by standard modules;
\item if $\Ext{A}{1}{\Delta_i}{\Delta_j}\neq 0$, then $i\lhd j$, for any choice of $j\in \Phi$.
\end{enumerate}
\end{defi}
\subsection{Exact Borel subalgebras}
\label{subsec:ebsubalgebras} 

An exact Borel subalgebra of a quasihereditary algebra $(A,\Phi, \unlhd)$ is a special subalgebra of $A$ that controls the standard modules and the category of all $A$-modules which are filtered by standard modules.

\begin{defi}[\cite{MR1362252}]
	\label{def:exactborel}
A subalgebra $B$ of a quasihereditary algebra $(A,\Phi, \unlhd)$ is an \emph{exact Borel subalgebra} of $A$ if the following hold:
\begin{enumerate}
\item the induction functor $A\otimes_B - :\MOD{B} \rightarrow \MOD{A}$ is exact;
\item $\Phi$ is an indexing set for the isomorphism classes of simple $B$-modules and $B$ is a quasihereditary algebra with respect to $(\Phi, \unlhd)$ having simple standard modules;
\item $A\otimes_B L_i^B =\Delta_i^A$ for every $i \in \Phi$;
\item $\End{B}{L_i^B}$ and $\End{A}{L_i^A}$ are isomorphic as $K$-vector spaces for every $i\in\Phi$.
\end{enumerate}
\end{defi}
\begin{rem}
In most of the literature that deals with exact Borel subalgebras it is assumed that the underlying field is algebraically closed. Therefore, condition (4) in Definition \ref{def:exactborel} is not usually included when defining an exact Borel subalgebra. However, in order for certain results in \cite{MR1362252} and \cite{MR3228437} to extend to the non-algebraically closed setting, it is necessary to add condition (4).
\end{rem}
Often, exact Borel subalgebras satisfy additional properties.
\begin{defi}[\cite{doi:10.1112/blms.12331}]
\label{defi:borelprops} 
Let $B$ be a subalgebra of some algebra $A$. Let $\Phi$ be an indexing set for the simple $B$-modules.
\begin{enumerate}
\item The subalgebra $B$ is \emph{normal} if the inclusion $\iota:B \hookrightarrow A$ has left inverse $\kappa$ as a morphism of right $B$-modules such that $\Ker{\kappa}$ is a right ideal of $A$.
\item The subalgebra $B$ is \emph{homological} if the functor $A\otimes_B - :\MOD{B} \rightarrow \MOD{A}$ is exact and, for every $X$ and $Y$ in $\Mod{B}$, the morphisms
\[\
\Ext{B}{n}{X}{Y} \longrightarrow \Ext{A}{n}{A\otimes_B X}{A\otimes_B Y},
\]
induced by $A\otimes_B - $, are isomorphisms for $n\geq 2$ and epimorphisms for $n=1$.
\item The subalgebra $B$ is \emph{regular} if the functor $A\otimes_B - :\MOD{B} \rightarrow \MOD{A}$ is exact and the morphisms
\[\
\Ext{B}{n}{L^B_i}{L^B_j} \longrightarrow \Ext{A}{n}{A\otimes_B L^B_i}{A\otimes_B L^B_j},
\]
induced by $A\otimes_B -$, are isomorphisms for every $n\geq 1$ and every $i,j\in \Phi$.
\end{enumerate}
\end{defi}
As observed in \cite[Remark 3.5]{doi:10.1112/blms.12331}, every normal and regular exact Borel subalgebra is homological. For the convenience of the reader, we show that every regular subalgebra $B$ of an algebra $A$ must be homological. The proof uses standard arguments from homological algebra.
\begin{lem}
	\label{lem:maini}
	If $B$ is a regular subalgebra of $A$, then $B$ is homological.
\end{lem}
\begin{proof}
Suppose that $B$ is a regular subalgebra of $A$. We claim that the morphisms
\[\
\psi_{X,Y}^n:\Ext{B}{n}{X}{Y} \longrightarrow \Ext{A}{n}{A\otimes_B X}{A\otimes_B Y},
\]
induced by the functor $A\otimes_B - $, are isomorphisms if $n\geq 2$ and epimorphisms if $n=1$, for every $X$ and $Y$ in $\Mod{B}$. We proceed by induction on $\ell(X)+\ell(Y)$ (here the operator $\ell$ denotes the composition length of a module). 

If $\ell(X)+\ell(Y)\leq 2$, then either $X$ and $Y$ are both simple modules or $\psi_{X,Y}^n$ is an isomorphism between trivial abelian groups. Since $B$ is regular, $\psi_{X,Y}^n$ must be an isomorphism if $X$ and $Y$ are simple. 

Assume now that $\ell(X)+\ell(Y)\geq 3$, so $\ell(X)\geq 2$ or $\ell(Y)\geq 2$. Suppose first that $\ell(Y)\geq 2$ and consider a short exact sequence
\begin{equation}
\label{eq:ses}
\begin{tikzcd}
0 \arrow[r] & Y' \arrow[r] & Y \arrow[r] & Y'' \arrow[r] & 0
\end{tikzcd}
\end{equation}
with $Y',Y''\neq 0$. Observe that $A\otimes_B-$ gives rise to a natural transformation $\psi^0_{X,-}$ from the functor $\Hom{B}{X}{-}:\MOD{B}\rightarrow\Ab$ to $\Hom{A}{A\otimes_B X}{A\otimes_B -}:\MOD{B}\rightarrow\Ab$. Since $A\otimes_B -$ is exact and $(\Ext{A}{n}{A\otimes_B X}{-})_{n\in\Znn}$ is a cohomological $\delta$-functor (see Definition 2.1.1 in \cite{wiebel}), then the composition of these two functors, namely
\[(\Ext{A}{n}{A\otimes_B X}{A\otimes_B -}: \MOD{B}\longrightarrow \Ab)_{n\in \Znn},\]
is also a cohomological $\delta$-functor. Given that $(\Ext{B}{n}{X}{-})_{n\in \Znn}$ is a universal cohomological $\delta$-functor from $\MOD{B}$ to the category of abelian groups (\cite[Definition 2.1.4 and §2.5.1]{wiebel}), then the natural transformation $\psi_{X,-}^0$ induces, for every $n\in \Zp$, a $2\times 5$ commutative diagram
\begin{equation*}
\begin{tikzcd}[column sep=normal]
 \Ext{B}{n-1}{X}{Y''} \arrow[d, "\psi^{n-1}_{X,Y''}"]\arrow[r] & \cdots \arrow[r]\arrow[d, phantom,"\cdots"] & \Ext{B}{n+1}{X}{Y'}\arrow[d, "\psi^{n+1}_{X,Y'}"] \\
\Ext{A}{n-1}{A\otimes_B X}{A\otimes_B Y''} \arrow[r] & \cdots \arrow[r] & \Ext{B}{n+1}{A\otimes_B X}{A\otimes_B Y'}
\end{tikzcd}.
\end{equation*}
By using induction and the Five Lemma (\cite[Chapter 1, Lemma $3.3$]{lane1995c}), we conclude that $\psi_{X,Y}^n$ is an isomorphism for every $n\geq 2$ and an epimorphism for $n=1$. The case $\ell(X)\geq 2$ can be proved in an analogous way, using the properties of contravariant cohomological $\delta$-functors.
\end{proof}

\subsection{Bocses} The research carried out in \cite{MR3228437} and \cite{doi:10.1112/blms.12331} revealed a close connection between quasihereditary algebras containing an exact Borel subalgebra and directed bocses.
\begin{defi}
	A \emph{bocs} is a quadruple $\mathfrak{B}=(B,W,\mu, \varepsilon)$ consisting of an algebra $B$ and a $B$-$B$-bimodule $W$ (possibly infinite-dimensional over $K$), together with a $B$-$B$-bilinear coassociative comultiplication $\mu: W \rightarrow W\otimes_B W$ and a $B$-$B$-bilinear counit $\varepsilon: W \rightarrow B$. 
\end{defi}

\begin{defi}[\cite{doi:10.1112/blms.12331,kulshammer2016bocs,MR3800074}]
	Let $\mathfrak{B}=(B,W,\mu, \varepsilon)$ be a bocs.
	\begin{enumerate}
		\item $\mathfrak{B}$ is \emph{normal} if there exists an element $w\in W$ such that $\mu(w)=w \otimes_B w$ and $\varepsilon(w)=1$ (such $w$ is called a \emph{group-like} element).
		\item $\mathfrak{B}$ is \emph{directed} if the counit $\varepsilon$ is epic and the following conditions hold:
		\begin{enumerate}
			\item $B$ is a quasihereditary algebra with respect to some indexing poset $(\Phi, \unlhd)$ and the standard $B$-modules are simple;
			\item $\Ker{\varepsilon}$ is a direct sum of finitely many $B$-$B$-bimodules of the form $Be_j \otimes_K e_i B$, with $i,j \in\Phi$ and $i\lhd j$.
		\end{enumerate}
	\end{enumerate}	
\end{defi}

To every bocs we may associate two algebras: its right and its left algebra. The right and left algebras of a bocs are not necessarily finite dimensional. In this note, only the notion of right algebra will be relevant.
\begin{defi}[{\cite{assforbocses}}]
	The \emph{right algebra} $R_{\mathfrak{B}}$ of a bocs $\mathfrak{B}=(B,W,\mu, \varepsilon)$ consists of the $B$-$B$-bimodule $\Hom{B}{W}{B}$ endowed with the multiplication $s\circ_{\mathfrak{B}}t$ for $s,t \in \Hom{B}{W}{B}$ given by the composition
	\[
	\begin{tikzcd}
	W \ar[r, "\mu"] & W\otimes_B W  \ar[r, "1_W\otimes_B s"] &W\otimes_B B \ar[r, "m_R^W"] & W \ar[r, "t"] & B
	\end{tikzcd}.
	\]
\end{defi}

\begin{rem}
Note that $W$ is finite-dimensional when $\mathfrak{B}=(B,W,\mu, \varepsilon)$ is directed. The right algebra of a directed bocs is therefore finite-dimensional.
\end{rem}

\section{New results}
Suppose that $B$ is an exact Borel subalgebra of a quasihereditary algebra $(A,\Phi, \unlhd)$. The algebra embedding $\iota:B \hookrightarrow A$ turns every module over $A$ into a module over $B$ by restriction of the action. We denote the restriction of the action on the left by $\Rest: \MOD{A} \rightarrow \MOD{B}$ and use the notation $\Rest': \MOD{A\op} \rightarrow \MOD{B\op}$ for the restriction functor on the right.

In order to prove our main result, the following theorem, due to Koenig, will be crucial.

\begin{thm}[{\cite[part of Theorem A]{MR1362252}}]
	\label{thm:koenigborel}
	Let $B$ be an exact Borel subalgebra of a quasihereditary algebra $(A,\Phi,\unlhd)$. The restriction functor $\Rest: \MOD{A} \rightarrow \MOD{B}$ gives rise to an isomorphism of $B$-modules $\Rest(\nabla_i^A)\cong Q_i^B= \nabla_i^B$.
\end{thm}

\begin{thm}
	\label{thm:main}
	Every exact Borel subalgebra of a quasihereditary algebra is normal.
\end{thm}
\begin{proof}
	Suppose that $B$ is an exact Borel subalgebra of a quasihereditary algebra $(A,\Phi, \unlhd)$, and regard the embedding  of $B$ into $A$ as a morphism $\iota:B \hookrightarrow \Rest '(A)$ of right $B$-modules. The opposite algebra $A\op$ of $A$ is still quasihereditary with respect to the same indexing poset $(\Phi, \unlhd)$ and $\Delta_i^{A\op}=D(\nabla_i^A)$ (see \cite[page 2]{MR1211481}). Using the standard duality and Theorem \ref{thm:koenigborel}, we get
	\[
	\Rest'\left( \Delta_i^{A\op}\right) \cong \Rest'\left( D\left( \nabla_i^A\right) \right) \cong D\left( \Rest \left( \nabla_i^A\right) \right) \cong D\left( Q_i^B\right) \cong P_i^{B\op}.
	\]
	Suppose that $B$ decomposes as $\bigoplus_{i\in \Phi}(P_i^{B\op})^{k_i}$ as a right $B$-module, for certain positive integers $k_i$ with $i\in\Phi$. There must exist some isomorphism
	\[\varphi:B\longrightarrow \Rest'\left( \bigoplus_{i\in \Phi}\left( \Delta_i^{A\op}\right)^{k_i} \right)\]
	in $\MOD{B\op}$. Let $\overline{\varphi}:A\rightarrow \bigoplus_{i\in \Phi}(\Delta_i^{A\op})^{k_i}$ be the morphism in $\MOD{A\op}$ given by $\overline{\varphi}(a)=\varphi(1_B)a$ for $a\in A$. Note that
\[
\varphi(b)=\varphi(1_B)b=\varphi(1_B)\iota(b)=\overline{\varphi}(\iota(b))=\left( \Rest'(\overline{\varphi})\circ \iota \right) (b).
\]
This means that the map $\iota$ is a split monic in $\MOD{B\op}$. Furthermore, the splitting epimorphism $\kappa=\Rest'(\overline{\varphi})$ can be realised as the restriction to $\MOD{B\op}$ of the epic $\overline{\varphi}$ in $\MOD{A\op}$. The kernel $\Ker{\overline{\varphi}}$ of $\overline{\varphi}$ is clearly a right ideal of $A$ and $\Rest'(\Ker{\overline{\varphi}})=\Ker{(\Rest'(\overline{\varphi}))}$.
\end{proof}

The techniques used in the proof of the main theorem in \cite{doi:10.1112/blms.12331} may now be applied to derive the following corollary of Theorem \ref{thm:main}.

\begin{cor}
	\label{cor:main}
	Every directed bocs is normal.
\end{cor}

\begin{proof}
Let $\mathfrak{B}=(B,W,\mu, \varepsilon)$ be a directed bocs. According to Theorem 11.2 in \cite{MR3228437} (see also Corollary 11.4 in \cite{MR3228437} and its proof), the right algebra $R_{\mathfrak{B}}$ of $\mathfrak{B}$ is quasihereditary and the morphism $\iota=\Hom{B}{\varepsilon}{B}$ is injective and turns $B$ into an exact Borel subalgebra of $R_{\mathfrak{B}}$; observe that the corresponding proofs are still valid even if $K$ is not algebraically closed. By Theorem \ref{thm:main}, $B$ is actually a normal exact Borel subalgebra of $R_{\mathfrak{B}}$, so $\iota$ splits as a morphism of right $B$-modules and one can choose a left inverse of $\iota$ whose kernel is a right ideal of $A$. Since $\mathfrak{B}$ is directed, then $W$ is finitely generated and projective as a left (and as a right) $B$-module, hence $R_{\mathfrak{B}}$ is projective as a right module over $B$ (\cite[§$2.1$]{assforbocses}). By the Dual Coring Theorem (see for instance \cite[Theorem 1]{10.2307/2044796}) and by Theorem 3 in \cite{10.2307/2044796}, the $B$-$B$-bimodule $\Hom{B\op}{R_{\mathfrak{B}}}{B}$ has a natural structure of a normal bocs over $B$ and $\Hom{B\op}{R_{\mathfrak{B}}}{B}\cong \mathfrak{B}$ as a bocs. Hence $\mathfrak{B}$ is normal.
\end{proof}

We conclude this note with a simpler formulation of the main result in \cite{doi:10.1112/blms.12331}.
Before stating this, the notion of regular directed bocs needs to be introduced.  
\begin{defi}
	A directed bocs $\mathfrak{B}=(B,W,\mu, \varepsilon)$ is \emph{regular} if $\Ext{B}{n}{\bigoplus_{i\in \Phi} L_i^B}{\bigoplus_{i\in \Phi} L_i^B}$ and $\Ext{A}{n}{\bigoplus_{i\in \Phi} A\otimes_B L_i^B}{\bigoplus_{i\in \Phi} A\otimes_B L_i^B}$ are isomorphic vector spaces (here $\Phi$ denotes an indexing set for the simple $B$-modules).
\end{defi}
\begin{rem}
This is not the original definition of regularity for directed bocses, but it is a succinct equivalent characterisation that follows from Lemma 3.12 in \cite{doi:10.1112/blms.12331} and Corollary \ref{cor:main}.
\end{rem}

As a consequence of Theorem \ref{thm:main}, Corollary \ref{cor:main} and also Lemma \ref{lem:maini}, we obtain a rephrasing of Theorem 3.13 in \cite{doi:10.1112/blms.12331}.
\begin{thm}[{\cite[Theorem 3.13]{doi:10.1112/blms.12331}}]
	\label{thm:mainlast}
	Let $K$ be an algebraically closed field. There is a one-to-one correspondence between directed bocses and quasihereditary algebras with a homological exact Borel subalgebra. This assignment restricts to a bijection between regular directed bocses and quasihereditary algebras with a regular exact Borel subalgebra.
\end{thm}

\bibliographystyle{amsplain}
\bibliography{QHAlg}

\providecommand{\bysame}{\leavevmode\hbox to3em{\hrulefill}\thinspace}
\providecommand{\MR}{\relax\ifhmode\unskip\space\fi MR }
\providecommand{\MRhref}[2]{%
  \href{http://www.ams.org/mathscinet-getitem?mr=#1}{#2}
}
\providecommand{\href}[2]{#2}
\begin{thebibliography}{10}

\bibitem{MR3800074}
A.~Bodzenta and J.~K\"{u}lshammer, \emph{Ringel duality as an instance of
  {K}oszul duality}, J. Algebra \textbf{506} (2018), 129--187.

\bibitem{doi:10.1112/blms.12331}
T.~Brzeziński, S.~Koenig, and J.~Külshammer, \emph{From quasi-hereditary
  algebras with exact {B}orel subalgebras to directed bocses}, Bull. Lond.
  Math. Soc. \textbf{52} (2020), no.~2, 367--378.

\bibitem{assforbocses}
W.~L. Burt and M.~C.~R. Butler, \emph{Almost split sequences for bocses},
  Representations of finite-dimensional algebras ({T}sukuba, 1990), CMS Conf.
  Proc., vol.~11, Amer. Math. Soc., Providence, RI, 1991, pp.~89--21.

\bibitem{MR1211481}
V.~Dlab and C.~M. Ringel, \emph{The module theoretical approach to
  quasi-hereditary algebras}, Representations of algebras and related topics
  ({K}yoto, 1990), London Math. Soc. Lecture Note Ser., vol. 168, Cambridge
  Univ. Press, Cambridge, 1992, pp.~200--224.

\bibitem{10.2307/2044796}
M.~Kleiner, \emph{The {D}ual {R}ing to a {C}oring with a {G}rouplike}, Proc.
  Amer. Math. Soc. \textbf{91} (1984), no.~4, 540--542.

\bibitem{MR1362252}
S.~Koenig, \emph{Exact {B}orel subalgebras of quasi-hereditary algebras. {I}},
  Math. Z. \textbf{220} (1995), no.~3, 399--426, With an appendix by Leonard
  Scott.

\bibitem{MR3228437}
S.~Koenig, J.~K{\"u}lshammer, and S.~Ovsienko, \emph{Quasi-hereditary algebras,
  exact {B}orel subalgebras, {$A_\infty$}-categories and boxes}, Adv. Math.
  \textbf{262} (2014), 546--592.

\bibitem{kulshammer2016bocs}
J.~K\"{u}lshammer, \emph{In the bocs seat: quasi-hereditary algebras and
  representation type}, Representation theory -- current trends and
  perspectives, EMS Ser. Congr. Rep., Eur. Math. Soc., Z\"{u}rich, 2017,
  pp.~375--426.

\bibitem{lane1995c}
S.~MacLane, \emph{Homology}, Classics in Mathematics, Springer Berlin
  Heidelberg, 1995.

\bibitem{wiebel}
C.~A. Weibel, \emph{An introduction to homological algebra}, Cambridge studies
  in advanced mathematics, Cambridge University Press, Cambridge, 1994.

\end{thebibliography}

\end{document}